\documentclass[11pt]{amsart}
\usepackage{graphicx}
\usepackage{commath}
\usepackage{stmaryrd}
\usepackage{xcolor}
\usepackage{bm}
\usepackage{hyperref}
\usepackage{mathabx}
\usepackage{tikz-cd}
\usepackage{enumitem}
\newcommand{\R} {\mathbb {R} }

\newcommand{\Ha}{\mathbb{H}}
\newcommand{\D} {\mathbb {D}}
\newcommand{\Sp} {\mathbb {S}}

\newcommand{\M}{\operatorname {M}}
\newcommand{\TB}{\mathcal {T}}

\newcommand{\K}{\mathcal {K}}

\newcommand{\flow}{{\mathbf{g}}}

\newcommand{\inj}{\mathbf{i}_{\M}}

\newcommand{\len}{\operatorname {l}}

\newcommand{\wh}{\widehat}
\newcommand{\dis}{\operatorname {d}}

\newtheorem{corollary}{Corollary}[section]
\newtheorem{theorem}[corollary]{Theorem}
\newtheorem{lemma}[corollary]{Lemma}

\newtheorem{definition}[corollary]{Definition}

\newtheorem{remark}[corollary]{Remark}

\begin{document}

\title[The Collar Lemma] {Simple geodesics in closed hyperbolic manifolds}

\author[Luo]{Qiliang LUO}

\address{\newline YMSC  \newline Tsinghua University   \newline Beijing, China }

\author[Markovi\'c]{Vladimir Markovi\'c}

\address{\newline YMSC  \newline Tsinghua University   \newline Beijing, China }

\today


\begin{abstract}  Given any closed hyperbolic n-manifold $\M$ ($n\ge 3$), we show that for a fixed $\kappa>\frac{3}{n-2}$, a random closed geodesic of length  close to $r$  has the collar of width $r^{-\kappa}$ when $r$ is large enough. In particular, most closed geodesics of length close to $r$ are simple. This theorem positively answers Problem 3.16 from Kirby's list \cite{b-k-r} which asks whether every closed hyperbolic 3-manifold contains infinitely many simple closed geodesics. \end{abstract}

\maketitle

\section{Introduction}

Throughout the paper $\M$ denotes a closed hyperbolic $n$-manifold for some $n\ge 3$. By $\Sp^T$ we denote the circle obtained as the quotient of the real line $\R$ by the translation $x\to x+T$, where $T>0$.

\begin{definition} A geodesic on $\M$ is a local isometry $\gamma:\Sp^{\len(\gamma)} \to \M$. Here $\len(\gamma)$ is the length of $\gamma$.  By  $\Gamma$ we denote the collection of  closed geodesics in $\M$.
\end{definition}

\begin{definition} Let $\dis_{\M}$  denote the hyperbolic distance on $\M$.  We say that $\gamma\in \Gamma$ is  $\delta$-collared if the implication
$$
\text{$\dis_{\M}(\gamma(x),\gamma(y))< \delta$ \quad $\implies$ \quad  
$\dis_{\M}(\gamma(x),\gamma(y))=\dis_{\Sp}(x,y)$}
$$ 
holds for every $x,y\in \Sp^{\len(\gamma)}$.  Here $\dis_{\Sp}(x,y)$ denotes the distance between $x$ and $y$ on circle $\Sp^{l(\gamma)}$.
\end{definition}
\begin{remark} A geodesic $\gamma\in \Gamma$ is  $\delta$-collared if and only if the collar of width $\frac{\delta}{2}$ around $\gamma$ is embedded in $\M$.
 \end{remark}

Suppose $0<\epsilon<r$, and $\delta>0$. We let $\Gamma_{r,\epsilon}$ denote the subset of $\Gamma$ consisting of closed geodesics whose length belongs to the interval $(r-\epsilon,r+\epsilon)$. By $\Gamma^\delta_{r,\epsilon}$ we denote the subset of 
$\Gamma_{r,\epsilon}$  consisting of geodesics which are not $\delta$-collared.  Our first main result states that a random geodesic in $\Gamma_{r,\epsilon}$ is $\delta$-collared with high probability for a  suitable $\delta$.

\begin{theorem}\label{thm-main} Suppose $\M$ is a closed hyperbolic n-manifold with $n\ge 3$. Let $\kappa>\frac{3}{n-2}$, and $\epsilon>0$. Then the following holds for $\delta=r^{-\kappa}$ 
$$
\lim_{r\to \infty} \frac{|\Gamma^{\delta}_{r,\epsilon}|}{|\Gamma_{r,\epsilon}|}=0.
$$
\end{theorem}

A closed geodesic $\gamma\in \Gamma$ is said to be simple if  it has no self intersections. In particular, such $\gamma$ is a prime geodesic. Note that $\gamma$ is simple if and only if it is $\delta$-collared for some $\delta>0$. The following theorem is an immediate consequence of Theorem \ref{thm-main}.

\begin{theorem}\label{thm-main-1} Every closed hyperbolic manifold contains infinitely many simple closed geodesics.
\end{theorem}
\begin{remark} Theorem \ref{thm-main} implies that a ''random" closed geodesic  in $\M$ is simple when the dimension of $\M$ is at least three. When $\M$ is a closed hyperbolic surface a random closed geodesic is not simple. However, in this case there are infinitely many simple closed geodesics on $\M$ as well. 
\end{remark}
This theorem positively answers Problem 3.16 from Kirby's list \cite{b-k-r} which asks whether every closed hyperbolic 3-manifold contains infinitely many simple closed geodesics (Kuhlmann \cite{kuhlmann} previously showed  this to be true for cusped hyperbolic 3-manifolds).

\subsection{Notation and conventions} In the remainder of this paper we fix $n\ge 3$, and a closed hyperbolic $n$-manifold $M$. It is important to stress that unless otherwise stated all constants in this paper depend only on $\M$.

\subsection{Tight unit vectors}
	 Denote by $\TB=T^1M$ the unit tangent bundle over $M$ and by $\pi : \TB\to M$ the corresponding projection. Consider the Sasaki Riemannian metric on $\TB$, and let $\dis_{\TB}$ denote the induced distance function on $\TB$. By $\Lambda$ we denote the Liouville (probability) measure on $\TB$. Let $\flow_r: \TB\to \TB$ be the geodesic flow. Then $\flow_r$ is a measure preserving diffeomorphism of the measure space $(\TB, \Lambda)$. Let $\inj$ denote the injectivity radius of $M$.

	\begin{definition}
		Let $r,\epsilon,\delta>0$. We say that $v\in\TB$ is $(r,\epsilon,\delta)$-tight if
		\[
		\dis_{\mathcal T}(v,\mathbf g_r(v))<\epsilon,
		\]
		and 
		 \[
		\dis_{\M}\left(\pi\circ \mathbf g_s(v),\pi\circ\mathbf g_t(v)\right)< \delta,
		\]
		for some $s,t\in [\tfrac{r}{10},\tfrac {9r}{10}]$ with $|s-t|\geq\mathbf i_{\M}$. The set of $(r,\epsilon,\delta)$-tight vectors is denoted by $\K^\delta_{r,\epsilon}$.
	\end{definition}
	
\begin{remark} If a vector $v\in \TB$ is tight then the geodesic segment of length $r$, with the initial tangent vector $v$, is not $\delta$-collared. This is closely related to the notion of non-collared closed geodesics in $\Gamma$ as we explain below.
\end{remark}	
	We show that the set of non-collared closed geodesics is statistically negligible by showing that the  same holds for the set $\K^\delta_{r,\epsilon}$. We do this by bounding above the measure of $\K^\delta_{r,\epsilon}$ using the Anosov properties of the flow $\mathbf g_r$, and basic hyperbolic geometry. The bulk of the paper is devoted to proving the following theorem.
	\begin{theorem}
	\label{thm: measure}
	There exist constants $C,\delta_0>0$ such that the  estimate
		\[
		\Lambda\left(\K^\delta_{r,\epsilon}\right)\leq C\delta^{n-2}r^2
		\]
		holds for every $r\geq 10\inj$,  every $0<\delta \le \delta_0$, and every $\epsilon>0$.
	\end{theorem}

	\subsection{Non-collared geodesics}
	
	The relationship between  non-collared closed geodesics and  tight vectors is built through the map $\mathcal J$ defined as follows. 
	
\begin{definition} Suppose $0<\epsilon<\inj$, and let  $r>0$. By $\K_{r,\epsilon}\subset \TB$ we denote the unit vectors $v$ such that $\dis_\TB(v,\flow_r(v))<\epsilon$. 
\end{definition} 	

For $v\in \K_{r,\epsilon}$ let $\alpha_r(v)$ be the length $r$ geodesic arc with the initial vector $v$. Since the endpoints of $\alpha_r(v)$ are at the distance at most $\epsilon$ there exists the unique shortest geodesic arc  in $\M$ joining these two endpoints.  Concatenating $\alpha$ with this shortest arc  forms the closed piecewise geodesic curve which we denote by $P_r(v) $. 

\begin{lemma}	There exists $\epsilon_0,r_0>0$ so that assuming $\epsilon\le \epsilon_0$, and $r\ge r_0$,  for every $v\in \K_{r,\epsilon}$ the piecewise geodesic curve $P_r(v)$ is homotopic to a closed geodesic $\mathcal P_{r,\epsilon}$ in $\M$.   This defines the map $\mathcal P_{r,\epsilon}:\K_{r,\epsilon}\to \Gamma$.
\end{lemma}
\begin{proof} When $\epsilon$ is small enough, for every large $r$ the curve $P_r(v) $ is a quasigeodesic (see the Chain Lemma in \cite{k-m})  and hence it is homotopic to a closed geodesic in $\M$.
\end{proof}		

The restriction of $\mathcal P_{r,\epsilon}$ to the set  $\K^\delta_{r,\epsilon}$ is denoted by $\mathcal J^\delta_{r,\epsilon}$ (although the definition of $\mathcal J^\delta_{r,\epsilon}$ does not depend on $\delta$, its domain does depend on $\delta$).
The significance  of the map $\mathcal J^\delta_{r,\epsilon}$ comes from Lemma \ref{lem: measure}. If $\gamma$ is a non-collared geodesic, then there are many tight vectors mapping to $\gamma$. This is quantified by the following lemma. For $\gamma\in \Gamma$, we let 
	\[
	\K^\delta_{r,\epsilon}(\gamma)=\mathcal (J^\delta_{r,\epsilon})^{-1}(\gamma)
	\] 
	denote the set of tight vectors mapping to $\gamma$.  
	\begin{lemma}
	\label{lem: measure}
		Let $\kappa,\epsilon>0$. Then there exist constants $r_0,C>0$ such that  for any $\gamma\in \Gamma^\delta_{r,\epsilon}$ the estimate 
		\[\Lambda\left(\K^{3\delta}_{r,3\epsilon}(\gamma)\right)\geq Ce^{-(n-1)r}\]
		holds when $r\geq r_0$, and  $\delta= r^{-\kappa}$.
	\end{lemma}
	
Combining the lower bound on $\Lambda\left(\K^{3\delta}_{r,3\epsilon}(\gamma)\right)$ with the upper bound on $\Lambda\left(\K^{3\delta}_{r,3\epsilon}\right)$, together with the Margulis' estimate on $|\Gamma_{r,\epsilon}|$, yields the proof of the main theorem. This is done in the next subsection.

	 \subsection{Proof of Theorem \ref{thm-main}} Fix $0<\epsilon<\inj$, and let $\delta=r^{-\kappa}$. It suffices to prove that 	 
	\[
	\lim_{r\to\infty} \frac{|\Gamma^\delta_{r,\epsilon}|}{|\Gamma_{r,\epsilon}|}=0.
	\]  
	 Let $C_1>0$ be the constant from Lemma \ref{lem: measure}. Then according to the same lemma, for any geodesic $\gamma\in \Gamma^\delta_{r,\epsilon}$ the inequality
	\[
	\Lambda\left(\K^{3\delta}_{r,3\epsilon}(\gamma)\right)\geq C_1e^{-(n-1)r},
	\]
	holds.  This implies
	\[
	\Lambda(\K^{3\delta}_{r,3\epsilon})\geq\sum_{\gamma\in\Gamma^\delta_{r,\epsilon}}\Lambda\left(\K^{3\delta}_{r,3\epsilon}(\gamma)\right)\geq C_1e^{-(n-1)r}|\Gamma^\delta_{r,\epsilon}|.
	\]
	Combining the above estimate with Theorem \ref{thm: measure}, we obtain
	\begin{equation}\label{eq-prevara}
	C_1|\Gamma^\delta_{r,\epsilon}|\le C_2 e^{(n-1)r}\delta^{n-2}r^2,
	\end{equation}
where  $C_2>0$ is the constant from Theorem \ref{thm: measure}. 

On the other hand,   it follows from the classical theorem by  Margulis \cite{margulis} that
	\[
|\Gamma_{r,\epsilon}|\ge C_3 \frac{e^{(n-1)r}}{r}
	\] for some constant $C_3>0$ independent of $r$ (the constant $C_3$ depends on $\epsilon$ though). Together with (\ref{eq-prevara}) this yields
$$
	  \frac{|\Gamma^\delta_{r,\epsilon}|}{|\Gamma_{r,\epsilon}|}\le C\delta^{n-2}r^3,
$$
where $C>0$ is a constant independent of $r$. Replacing $\delta=r^{-\kappa}$ in the above probability estimate, we obtain the upper bound 
$$
	  \frac{|\Gamma^\delta_{r,\epsilon}|}{|\Gamma_{r,\epsilon}|}\le Cr^{(3-(n-2)\kappa)}.
$$	
Since $\kappa>\tfrac{3}{n-2}$, this upper bound tends to zero as $r\to\infty$, which proves the theorem.

	\subsection{Outline and organization}	 
	 We prove Lemma \ref{lem: measure} in Section 2. The proof is based on showing that the 
set $\K^{3\delta}_{r,3\epsilon}(\gamma)$ contains a suitable Bowen ball which  is a dynamically defined neighbourhood of a unit vector tangent to $\gamma$. The lower bound in   Lemma \ref{lem: measure} then steams from the standard computation of the volume of a Bowen ball.

We prove Theorem \ref{thm: measure} in Sections 3,4, and 5. In Section 3 we define the set $\mathcal L^\delta_r\subset \TB$ and show that the total measure of $\mathcal K^\delta_{r,\epsilon}$ is in a suitable sense dominated by the total  measure of $\mathcal L^\delta_r$. The set $\mathcal L^\delta_r$ eliminates the dependence on $\epsilon$ and is much simpler to work with than  $\mathcal K^\delta_{r,\epsilon}$.

In Section 4 we recall some standard estimates from hyperbolic geometry, and the orbit counting estimate on the universal cover of a hyperbolic manifold. In the final Section 5 we finally compute the upper bound on the total measure of $\mathcal L^\delta_r$.

\section{Tight vectors near non-collared geodesic}

This section is devoted to proving Lemma \ref{lem: measure} which provides a lower bound on the measure of tight vectors near a non-collared geodesic $\gamma$. To do so, we find a single tight vector near $\gamma$ and study its \emph{dynamical neighbourhood} called the Bowen ball. The proof of Lemma \ref{lem: measure} is  completed in Subsection \ref{sec: proof lemma}.

\subsection{Dynamical neighbourhoods}
The geodesic flow $\flow_t: \TB \to \TB$ defines a smooth $\R$-action on the manifold $\TB$. 
Because $\M$ has constant negative sectional curvature the flow is Anosov, and the manifold $\TB$ is equipped with three mutually transverse $\flow_t$-invariant foliations. Suppose that $W^s$ is a local strong stable leaf, $W^u$ a local strong unstable leaf, and $W^c$ a central flow leaf. Then there exists a constant $C>0$ such that  
\begin{itemize}
 \item  For any $v_1, v_2 \in W^s$, and all $t \ge 0$, we have
$$
\dis_{\TB}(\flow_t(v_1), \flow_t(v_2)) \le C e^{-t} \dis_\TB(v_1, v_2).
$$
\vskip .1cm
\item For any $v_1, v_2 \in W^u$, and all $t \ge 0$, we have
$$
\dis_\TB(\flow_t(v_1), \flow_t(v_2)) \ge C^{-1} e^t \dis_\TB(v_1, v_2).
$$
\vskip .1cm
\item If $v_1, v_2\in W^c$  then $v_2=\flow_s(v_1)$ for some $s$. 
\end{itemize}
\begin{definition}\label{def-bowen} The Bowen ball $B_r(v, \epsilon)$ consists of all vectors $u\in \TB$ such that  
$$
\dis_\TB(\flow_t(u),\,  \flow_t(v)) < \epsilon,
$$ 
for all $t \in [0, r]$. We call $B_r(v,\epsilon)$ the $(r,\epsilon)$-Bowen ball centred at $v$.
\end{definition}
Due to the leaf dynamics, it forms a product cylinder
$$
B_r(v, \epsilon) \cong \D^{n-1}_{\epsilon e^{-r}}(W^u) \times \D^{n-1}_{\epsilon}(W^s) \times (-\epsilon, \epsilon).
$$
The following two lemmas are standard and can be essentially found in \cite{k-h}. The next one follows from 
Lemma 20.1.1 in \cite{k-h}.
\begin{lemma}\label{lemma-mera} For any $r_0>0$ there exist constants $C>1$ and $\epsilon_0>0$ such that for every $r>r_0$ and every $\epsilon_0\geq \epsilon>0$ we have
$$
C^{-1} \epsilon^{2n-1}e^{-(n-1)r}\le \Lambda(B_r(v, \epsilon))\le C \epsilon^{2n-1}e^{-(n-1)r}.
$$
\end{lemma}

\begin{lemma}\label{lemma-dara} Let $\epsilon>0$, and $0<q<1$. There exists a constant $C>0$ so that for every  $w \in B_r(v, \epsilon)$  there exists $c\in (-\epsilon,\epsilon)$ with 
$v_c = \flow_c(v)$ living in the central leaf of $v$, and  such that 
$$
\dis_\TB(\flow_t(w), \flow_t(v_c)) \le C e^{-qr}
$$
for every $qr<t<(1-q)r$.
\end{lemma}

\begin{proof} For each  $w \in B_r(v, \epsilon)$ there exists a unique $v_c = \flow_c(v)$ such that $w$ belongs 
to the  transverse section (or the transverse slice) to the geodesic flow over $v_c$. This transverse section 
is of dimension $(2n-2)$, and has the local product structure 
$\D^{n-1}_{\epsilon e^{-r}}(W^u) \times \D^{n-1}_{\epsilon}(W^s)$. 

In particular, there exists a vector $u$ in the Bowen ball  $B_r(v, \epsilon)$ such that $u\in W^{s}(w)\cap W^u(v_c)$. Then 
$$
\dis_\TB(\flow_t(w),\flow_t(u))\le Ce^{-t}\dis_\TB(w,u)\le Ce^{-t},
$$
and
$$
\dis_\TB(\flow_t(u),\flow_t(v_c))\le Ce^{t-r}\dis_\TB(\flow_r(u),\flow_r(v_c))\le Ce^{t-r}.
$$ 
The estimate in the lemma then follows from the above contraction/expansion properties of the flow's action on $W^s$ and $W^u$.
\end{proof}

\subsection{Relation with the set $\K^\delta_{r,\epsilon}$} The next lemma shows that the set  $\K^{3\delta}_{r,3\epsilon}(\gamma)$ contains a suitable Bowen ball. This enables us to bound below the measure of $\K^{3\delta}_{r,3\epsilon}(\gamma)$ which is the goal of Lemma \ref{lem: measure}.

\begin{lemma}\label{lemma-nena}
For each $\kappa>0$ there exist $\epsilon_0, r_0>0$ with the following properties. Given $\gamma\in \Gamma$, suppose  $\flow_c(v) \in \K^\delta_{r,\epsilon}(\gamma)$ for every $c\in (-\epsilon,\epsilon)$. Then
$$
B_r(v, \epsilon)\subset \K^{3\delta}_{r,3\epsilon}(\gamma),
$$
for  $\delta=r^{-\kappa}$, where $0<\epsilon\le \epsilon_0$, and $r \ge r_0$.
\end{lemma}

\begin{proof} Let  $w \in B_r(v, \epsilon)$.  By Definition \ref{def-bowen} we have that 
$$
\dis_\TB(w,v), \, \dis_\TB(\flow_r(w),\flow_r(v))<\epsilon.
$$ 
Since $v\in \K^\delta_{r,\epsilon}$ it follows that $\dis_\TB(v,\flow_r(v))<\epsilon$. Combining these we conclude that 
\begin{equation}\label{eq-dimit}
\dis_\TB(w,\flow_r(w))<3\epsilon.
\end{equation}
On the other hand, by Lemma \ref{lemma-dara}  there exists $c\in (-\epsilon,\epsilon)$ such that 
\begin{equation}\label{eq-range}
\dis_\TB(\flow_t(w), \flow_t(v_c)) \le C_1 e^{-\frac{r}{10}}
\end{equation}
for every $\frac{r}{10}\le t \le \frac{9r}{10}$, where $v_c = \flow_c(v)$, and where $C_1$ is the constant from  Lemma \ref{lemma-dara}. Since $v\in \K^\delta_{r,\epsilon}$ there exist  $s,t\in [\tfrac{r}{10},\tfrac {9r}{10}]$ with $|s-t|\geq\mathbf i_{\M}$, and such that 
$$
\dis_{\M} (\pi\circ \flow_s(v),\pi\circ \flow_t(v))< \delta.
$$
Applying (\ref{eq-range}) to this particular $t$ and $s$ we conclude that
$$
\dis_{\M} (\pi\circ \flow_s(w),\pi\circ \flow_t(w))< \delta +2C_1 e^{-\frac{r}{10}}.
$$
Since $\delta=r^{-\kappa}$, we can find $r_0$ large enough so that  $\delta +C_1e^{-\frac{r}{10}}<2\delta$. This shows that
$B_r(v, \epsilon)\subset \K^{3\delta}_{r,3\epsilon}$. But for any two vectors $u,w\in B_r(v, \epsilon)$ the pointwise distance between piecewise geodesics $P_r(u)$ and $P_r(w)$ is smaller than $\epsilon$ which implies that they are homotopic to small enough $\epsilon$. Thus, $\mathcal J^\delta_{r,\epsilon}(u)=\mathcal J^\delta_{r,\epsilon}(w)$, and $B_r(v, \epsilon)\subset \K^{3\delta}_{r,3\epsilon}(\gamma)$ where $\gamma=\mathcal J^\delta_{r,\epsilon}(v)$.
\end{proof}

\subsection{Proof of Lemma \ref{lem: measure}}\label{sec: proof lemma} 
Fix  $\kappa>0$, and let  $\delta= r^{-\kappa}$ where $r>0$. Suppose $\epsilon\le \epsilon_0$ where $\epsilon_0$ is the constant from Lemma \ref{lemma-nena}.  Let $\gamma\in\Gamma^\delta_{r,\epsilon}$. In view of Lemma \ref{lemma-nena} and Lemma \ref{lemma-mera},  it suffices to show that there exists $v \in \TB$ such that  $\flow_c(v) \in \K^\delta_{r,\epsilon}(\gamma)$ for every $c\in (-\epsilon,\epsilon)$.

Recall the map $\gamma: \Sp^{\len(\gamma)}\to \M$.  Since $\gamma$ is not $\delta$-collared there exist two points $x,y\in \Sp^{\len(\gamma)}$ such that $\dis_{M}(\gamma(x),\gamma(y))<\delta$ and $\dis_{M}(\gamma(x),\gamma(y)) \neq \dis_\Sp(x,y)$. Let $I\subset \Sp^{\len(\gamma)}$ denote  the longer  subarc of $\Sp^{\len(\gamma)}$ with endpoints $x$ and $y$. Then  $\len(\gamma(I))\geq\tfrac{r}{4}$. 

Let $z\in \Sp^{\len(\gamma)}$ be the midpoint of the subarc $I$, and let  $v\in \TB$ be the  vector tangent to $\gamma$ at the point $\gamma(z)$. Then clearly 
$$
\dis_{\mathcal T}(v,\mathbf g_r(v))<\epsilon,
$$
and 
$$
\dis_{\M} \left(\pi\circ \flow_s(v),\pi\circ \flow_t(v) \right)< \delta,
$$
for some $s,t\in [\tfrac{r}{4},\tfrac {4r}{4}]$ with $|s-t|\geq\mathbf i_{\M}$. Furthermore, when $r$ is large enough the same holds for every $v_c=\flow_c(v)$, $c\in(-\epsilon,\epsilon)$. This enables us to apply Lemma \ref{lemma-nena} and the proof is complete.

\section{The measure of the set of tight vectors}
In this section, we prove Theorem \ref{thm: measure} by bounding above the measure of tight vectors $\K^\delta_{r,\epsilon}$. 
\begin{definition} Let $\delta>0$. By $\mathcal L^\delta_r$ we denote the set consisting of all $v\in \TB$ satisfying
$$
\dis_{\M}(\pi(v),\pi\circ\flow_t(v))\leq\delta,
$$
for some $\inj\leq t\leq r$.
\end{definition}
First we estimate the measure of $\mathcal L^\delta_r$. Then we complete the proof of Theorem \ref{thm: measure} by showing that  the total measure of $\mathcal K^\delta_{r,\epsilon}$ is dominated by the total  measure of $\mathcal L^\delta_r$.

\begin{lemma}\label{cor: L-estimate} There exist  constants $C,\delta_0>0$ such that  the inequality
$$
\Lambda\left(\mathcal L^\delta_r\right)\leq C\delta^{n-1}r
$$
holds for every $r\ge 10\inj$, and every  $0<\delta\le \delta_0$.
\end{lemma}

\subsection{Proof of Theorem \ref{thm: measure}}
We estimate the measure of $\K^\delta_{r,\epsilon}$ from above  by the measure of $\mathcal L^{2\delta}_r$. Fix $v\in\K^\delta_{r,\epsilon}$. By definition,  there exist suitable times 
$$
\frac{r}{10}\le t,s \le \frac{9r}{10},
$$ 
and
$$
t\ge s+\inj,
$$
such that 
\begin{equation}\label{eq-vaza}
\dis_{\M}\left(\pi\circ \mathbf g_s(v),\pi\circ\mathbf g_t(v)\right)< \delta.
\end{equation}
Let $k=\lfloor\tfrac{s}{\delta}\rfloor$. Then $|s-k\delta|\le \delta$.	 Since $\flow_t$ is the geodesic flow,  we have 
	\[
	\dis_{\M}\left(\pi\circ \mathbf g_{k\delta}(v),\pi\circ\mathbf g_s(v)\right)\leq |s-k\delta|\leq\delta.
	\]
	Combining this with ({\ref{eq-vaza}), and using the triangle inequality, we obtain the  estimate
	\[
	\dis_{\M}\left(\pi\circ \mathbf g_{k\delta}(v),\pi\circ\mathbf g_t(v)\right)< 2\delta,
	\] 
	which implies that $v\in \flow_{k\delta}^{-1}(\mathcal L^{2\delta}_r)$. Hence, we have established the inclusion
\begin{equation}\label{eq-vaza-1}
	\K^\delta_{r,\epsilon}\subset\bigcup_{k=0}^{\lfloor\tfrac{r}{\delta}\rfloor}\flow_{k\delta}^{-1}(\mathcal L^{2\delta}_r).
\end{equation}

Since the flow $\flow_t$ preserves the measure, we have
	$$
	\Lambda\left(\K^\delta_{r,\epsilon}\right)\leq\sum_{k=0}^{\lfloor\tfrac{r}{\delta}\rfloor}\Lambda\left(\flow_{k\delta}^{-1}(\mathcal L^{2\delta}_r)\right)\leq \frac{r}{\delta} \, \Lambda\left(L^{2\delta}_r\right).
	$$
Combining the above estimate with Lemma \ref{cor: L-estimate}, we obtain
\[
\Lambda\left(\K^\delta_{r,\epsilon}\right)\leq C\delta^{n-2}r^2,
\]
for some constant $C>0$. This completes the proof.

\section{Hyperbolic geometry background}

In this section we obtain some preliminary results needed in the proof of  Lemma \ref{cor: L-estimate}. First, we estimate above the Sasaki distance between vectors that are tangent to two fellow traveling geodesic arcs in the hyperbolic space $\Ha^n$. Then we recall the notion of a $\delta$-net in $\M$. We finish the section by recalling the standard orbit counting estimate and prove it's corollary.

\subsection{Estimating the Sasaki distance}

\begin{lemma}\label{lem: dynamical nbhd-1} For each $L>0$ there exists a constant $C=C(L)$ so that for every $\delta>0$ the following holds. 
Suppose $\alpha,\beta$ are geodesic arcs in $\Ha^n$ of length at least $L$, and whose endpoints are at most $\delta$ apart. If $u_j(\alpha)$ and $u_j(\beta)$, $j=1,2$, are the initial and the terminal tangent vectors at $\alpha$ and $\beta$ respectively then 
$$
\dis_{\TB}(u_j(\alpha),u_j(\beta))\leq C\delta.
$$
\end{lemma}
\begin{proof}
Let $\Theta(\cdot,\cdot)$ denote the non-oriented angle between two unit vectors supported at the same point.  Let $a,b\in \Ha^n$.  If $u\in \TB_a(\Ha^n)$ then $u@ b\in \TB_b(\Ha^n)$ denotes the vector parallel transported to $b$ along the geodesic segment connecting $a$ and $b$.  For $u\in \TB_a(\Ha^n)$, and $v\in \TB_b(\Ha^n)$, we have the following fact:
	\begin{equation}
	\label{E000}
	 \dis_{\TB}(u,v)\leq\dis_{\Ha^n}(a,b)+\Theta(v@b,u).
	\end{equation}

Denote by $a_1$ and $a_2$ the endpoints of $\alpha$, and by $b_1$ and $b_2$ the endpoints of $\beta$, such that 
 \begin{equation}\label{E000.5}
\dis_{\Ha^{n}}(a_j,b_j)\le \delta.
\end{equation}
Let $\gamma$ be the geodesic arc in $\Ha^n$ connecting $a_1$ and $b_2$.  By $u_1(\alpha),u_1(\beta),u_1(\gamma)$, and $u_2(\alpha),u_2(\beta),u_2(\gamma)$, we denote respectively the initial and the terminal tangent vectors of corresponding geodesic arcs.

The geodesic arcs $\alpha$ and $\gamma$ form a narrow triangle in hyperbolic space with two long edges of length at least  $(L-\delta)$, and one short edge of length at most $\delta$. By the cosine rule for hyperbolic triangles there exists a constant $C_1=C_1(L)$ such that  (see also \cite{e-m-m})
\begin{equation}\label{E001}
\Theta (u_1(\alpha),u_1(\gamma))\leq C_1\delta.
\end{equation}
Furthermore, according to  Claim 4.1 in \cite{k-m}, we have
$$
\Theta (u_2(\alpha),u_2(\gamma)@a_2)\leq \Theta (u_1(\alpha),u_1(\gamma))+\dis_{\Ha^{n}}(a_2,b_2)\le (1+C_1)\delta,
$$
where the last inequality follows from (\ref{E001}) and (\ref{E000.5}).
Combining this with (\ref{E000}) and (\ref{E000.5}),  we obtain
$$
\dis_{\TB}(u_2(\alpha),u_2(\gamma))\leq \Theta (u_2(\alpha),u_2(\gamma)@a_2)+\dis_{\Ha^n}(a_2,b_2)\leq (2+C_1)\delta.
$$
Applying the same argument to the narrow triangle with the edges $\beta$ and $\gamma$, same as in (\ref{E001})
we obtain the inequality
$$
\dis_{\TB}(u_2(\beta),u_2(\gamma))= \Theta(u_2(\beta),u_2(\gamma)) \leq C_1\delta.
$$	
By the triangle inequality, 
$$
\dis_{\TB}(u_2(\alpha),u_2(\beta)) \le \dis_{\TB}(u_2(\alpha),u_2(\gamma))+ \dis_{\TB}(u_2(\beta),u_2(\gamma)) \le 2(1+C_1)\delta.
$$
The similar argument bounds above the distance $\dis_{\TB}(u_1(\alpha),u_1(\beta))$. This  completes the proof by letting $C=2(1+C_1)$.
\end{proof}

We have the following corollary of the previous lemma.

\begin{lemma}\label{lem: dynamical nbhd} For each $L>0$ there exists a constant $c=c(L)$ so that for every 
$0<\delta\le L/4$ the following holds. 
Suppose $\alpha,\beta$ are geodesic arcs in $\Ha^n$ of length at least $L$, and whose endpoints are at most $\delta$ apart. Let $u$ and $v$ denote the initial tangent vectors to $\alpha$ and $\beta$ respectively. Then 
$u\in B_{\len(\beta)}(v,c\delta)$ (recall that $B_{\len(\beta)}(v,c\delta)$  denotes the $(\len(\beta),c\delta)$-Bowen ball centred at $v$).
\end{lemma}
\begin{proof} Let $\wh{\alpha}$ be the geodesic arc obtained by either extending or shortening $\alpha$, and such that
$\len(\wh{\alpha})=\len(\beta)$. Then 
\begin{itemize}
\item  $\alpha$ and $\wh{\alpha}$ have the same initial tangent vector $u$, 
\vskip .1cm
\item  $\len(\wh{\alpha})=\len(\beta)\ge L-\delta \ge \frac{3L}{4}$,
\vskip .1cm
\item the distance between the endpoints of $\wh{\alpha}$ and $\beta$ is at most $2\delta$.
\end{itemize}
Consider the natural parameterisations $\wh{\alpha}, \beta : [0,r] \to \Ha^n$, where $r=\len(\beta)$.
Let $\wh{\alpha}_0(t)$ and $\beta_0(t)$ be the geodesic arcs with the endpoints $\wh{\alpha}(0),\wh{\alpha}(t)$, and 
$\beta(0),\beta(t)$, respectively. Let $\wh{\alpha}_1(t)$ and $\beta_1(t)$ be the geodesic arcs with the endpoints $\wh{\alpha}(t),\wh{\alpha}(r)$, and $\beta(t),\beta(r)$, respectively. Then there exists $j\in \{0,1\}$ such that  both geodesic arcs $\wh{\alpha}_j$ and $\beta_j$ are longer than $L/3$, and that the distance between the endpoints of 
$\wh{\alpha}_j$ and $\beta_j$ is at most $2\delta$ (the last claim follows from the convexity of the hyperbolic distance).

Let $u(t)$ and $v(t)$ be the vectors tangent to $\wh{\alpha}(t)$ and $\beta(t)$. From the previous lemma (applied to $\wh{\alpha}_j$ and $\beta_j$) we find that 
$$
\dis_{\TB}(u(t),v(t))\leq 2C_1\delta
$$
where $C_1=C_1(L/3)$ is the constant from Lemma \ref{lem: dynamical nbhd-1}. Letting $C=2C_1$ proves the lemma.
\end{proof}

\subsection{The $\delta$-nets in $\M$} We recall the definition of a $\delta$-net.

\begin{definition}\label{def-net} Let $\mathcal P$ be a finite collection of points in $\M$, and $\delta>0$. We say that $\mathcal P$ is a $\delta$-net  if
\begin{itemize}
\item[(a)] for any $p\in \M$ there exists $p_j \in \mathcal{P}$ such that $\dis_{\M}(p,p_j)< \delta$,
\vskip .1cm
\item[(b)] for any $p_1,p_2\in\mathcal P$ we have  $ \dis_{\M}(p_1,p_2)\ge \delta$. 
\end{itemize}
\end{definition}
\begin{lemma}\label{lemma-mozda} There exists a constant $C>0$ such that for each $\delta>0$ there exists a  $\delta$-net $\mathcal P$ in $\M$ with the property
\begin{equation}\label{E004}
|\mathcal P|\leq \frac{\text{Vol}(\M)}{\text{Vol(radius $\delta/2$ ball)}}\leq C\delta^{-n}.
\end{equation}
\end{lemma}
\begin{proof}
The net $\mathcal P$ can be found by considering a maximal subset of $\M$ satisfying the  property (b). We estimate the number of points in $\mathcal P$ as follows. To each point $p\in\mathcal P$ we assign the radius $\delta/2$-ball centred at $p$. Since $\mathcal P$ is $\delta$-separating (property (b) in Definition \ref{def-net}), these balls are disjoint. Therefore, there exists a constant $C>0$ such that (\ref{E004}) holds.
\end{proof}

\subsection{Pointed geodesic loops}

Fix $p\in \M$. By $\Pi_p$ we denote the collection of geodesic arcs which start and end at $p$ (we call them pointed geodesic loops). Given a pointed geodesic loop $\gamma \in \Pi_p$,
we let $B^\delta_\gamma$ denote the  $(\len(\gamma),\delta)$-Bowen ball centred at the initial unit tangent vector of $\gamma$ (see Definition \ref{def-bowen}).

\begin{lemma}\label{lemma-proba} There exist constants $c,\delta_0>0$ such that for every $\gamma\in\Pi_p$, and every $0<\delta\le \delta_0$, we have
\begin{equation}\label{E-mera}	
\Lambda\left(B^{\delta}_\gamma\right)\leq C\delta^{2n-1}e^{-(n-1)\len(\gamma)}.
\end{equation}
\end{lemma}

\begin{proof}
As a closed loop $\gamma$ is   homotopically nontrivial in $\M$. Thus,  we have $\len(\gamma)\geq \inj$.
By letting $r_0=\inj$ in Lemma \ref{lemma-mera}, we  obtain constants $C,\delta_0>0$ such that (\ref{E-mera}) holds.
\end{proof}

\begin{lemma}\label{lemma-proba-1} There exists a constant $C>0$ such that the estimate
\begin{equation}\label{eq-proba-1}
\sum_{\substack{\gamma\in\Pi_{p}\\ 0<\len(\gamma)\leq r}}e^{-(n-1)\len(\gamma)}\le Cr
\end{equation}
holds for every $r>0$.
\end{lemma}

\begin{proof} Let $\Pi_p(s_1,s_2)$ be the subset of $\Pi_p$ consisting of pointed geodesic loops whose length lives in the interval $(s_1,s_2]$. By $N_{s_{1}}^{s_{2}}$ we denote the number of loops in  $\Pi_p(s_1,s_2)$.
Using the classical orbit counting estimate proved by Sullivan (see Corollary 5 in  \cite{sullivan}), we know that 
\begin{equation}\label{eq-sullivan}
N_s^{s+1}\le C_1e^{(n-1)s},
\end{equation}
for some constant $C_1>0$, and every  $s>0$. 

Let $r>\inj$. We have
\begin{align*}
\sum_{\gamma\, \in \, \Pi_{p}(0,r)}e^{-(n-1)\len(\gamma)}&\le\sum_{s=0}^{\lfloor r \rfloor}\,\, \sum_{\gamma \, \in \, \Pi_{p}(s,s+1)}e^{-(n-1)\len(\gamma)} \\
&\le \sum_{s=0}^{\lfloor r \rfloor}\,\,\sum_{\gamma\, \in \, \Pi_{p}(s,s+1)}e^{-(n-1)s} \\
&\le \sum_{s=0}^{\lfloor r \rfloor} N_s^{s+1} e^{-(n-1)s}\le \sum_{s=0}^{\lfloor r \rfloor} C_1 \\
&\le (1+r)C_1\le Cr, 
\end{align*}
where the third inequality follows from (\ref{eq-sullivan}), and the last inequality follows by letting 
$$
C=C_1\left(1+\frac{1}{\inj}\right).
$$
\end{proof}

\section{Proof of Lemma \ref{cor: L-estimate}}

We first decompose $\mathcal L^\delta_r$ into several Bowen balls. Given a finite set $\mathcal P\subset \M$, we let $\Pi_{\mathcal P}$ be the collection of pointed geodesic loops 
$$
\Pi_{\mathcal P}=\bigsqcup_{p\in\mathcal P}\Pi_p.
$$
We begin by proving the following lemma.
\begin{lemma}\label{lem: disintegration}  There exists a universal constant $c>0$ such that the following holds. Let $0<\delta \le \inj/10$, and $r\ge 10\inj$. Then for every $\delta$-net  $\mathcal P$  we have
$$
\mathcal L^\delta_r\subset \bigcup_{\substack{\gamma\in\Pi_{\mathcal P}\\ 0<\len(\gamma)\leq 2r}}B_\gamma^{c\delta}.
$$
\end{lemma}
\begin{proof} Fix $v \in \mathcal  L^\delta_r$. Since $\mathcal P$ is a $\delta$-net of $M$, there exists a point $p\in\mathcal P$ such that 
$$
\dis_{\M}(\pi(v),p)\leq\delta.
$$ 
By the definition of $\mathcal L^\delta_r$, we have $\dis_{\M}(\pi(v),\pi\circ\flow_{t}(v))\leq\delta$ for some time $t\in [\inj,r]$. By triangle inequality this implies
	\begin{equation}
	\label{E003}
	\dis_{\M}(\pi(v),p),\, \, \dis_{\M}(\pi\circ\flow_t(v),p)\leq 2\delta.
	\end{equation}
Let $\alpha$ denote the length $t$ geodesic arc with initial tangent vector $v$. By $\alpha_1$ we denote the closed loop obtained by concatenating $\alpha$ with the shortest geodesic arcs connecting the two endpoints of $\alpha$ with the point $p$ respectively (since the distance between the endpoints of $\alpha$ and the point $p$ is smaller than $\inj/2$, it follows that such shortest arc is unique). By $\beta$ we denote the pointed geodesic loop (pointed at $p$) which is homotopic to $\alpha_1$ (relative the point $p$). Note that 
$$
\len(\beta)\le \len(\alpha)+\delta+2\delta\le r+\frac{3\inj}{10}<2r.
$$

Now, replacing the inequality   (\ref{E003}) in  Lemma \ref{lem: dynamical nbhd}, we find that
$v\in B^{c\delta}_\beta$,  where $c=2C(10\inj)$ (here $C$ is the constant from Lemma \ref{lem: dynamical nbhd}). 
Since $\len(\beta)<2r$,  the proof is complete.
\end{proof}

\subsection{Proof of Lemma \ref{cor: L-estimate}} Let $\delta_1>0$ be the constants from Lemma \ref{lemma-proba}, and suppose 
$$
0<\delta \le \min\left\{\delta_1,\frac{\inj}{10}\right\}=\delta_0.
$$ 
Let $\mathcal P$ be a $\delta$-net satisfying the conclusion of  Lemma \ref{lemma-mozda}. According to Lemma \ref{lem: disintegration}  we have 
\begin{equation}\label{eq-moram}
\Lambda\left(\mathcal L^\delta_r\right)\leq\sum_{p\in\mathcal P}\sum_{\substack{\gamma\in\Pi_{p}\\ 0<\len(\gamma)\leq 2r}}\Lambda\left(B_\gamma^{c\delta}\right)
\end{equation}
for every $r\ge 10\inj$, where $c$ is the constant from Lemma \ref{lem: disintegration}.

Let $C_1$ be the constant from Lemma \ref{lemma-proba}, and $C_2$ the constant from Lemma \ref{lemma-proba-1}. For any $p\in\mathcal P$ the following holds
\[
\sum_{\substack{\gamma\in\Pi_{p}\\ 0<\len(\gamma)\leq 2r}}\Lambda\left(B_\gamma^{c\delta}\right)\leq C_1c^{2n-1}\, \delta^{2n-1}\sum_{\substack{\gamma\in\Pi_{p}\\ 0<\len(\gamma)\leq 2r}}e^{-(n-1)\len(\gamma)}\leq C_1C_2c^{2n-1}\, \delta^{2n-1}r,
\]
where the first inequality follows from (\ref{E-mera}), and the second from (\ref{eq-proba-1}). Combining this with (\ref{eq-moram}) implies
\[
\Lambda\left(\mathcal L^\delta_r\right)\leq |\mathcal P|\, C_1C_2c^{2n-1} \, \delta^{2n-1}r\leq C_1C_2C_3c^{2n-1}\, \delta^{n-1}r,
\]
where $C_3$ is the constant from Lemma \ref{lemma-mozda}. 
Letting $C=C_1C_2C_3c^{2n-1}$ completes the proof.

\end{document}